\documentclass[11pt]{amsart}

\pdfoutput=1

\usepackage[T1]{fontenc}
\usepackage{amsmath,amssymb}
\usepackage[a4paper,top=3cm,bottom=3cm,inner=3.5cm,outer=2.7cm,
            nofoot,headsep=1.5cm]{geometry}
\usepackage[colorlinks=true,linkcolor=blue,citecolor=blue,urlcolor=blue]{hyperref}
\usepackage{doi}

\theoremstyle{plain}
\newtheorem{theorem}{Theorem}[section]
\newtheorem{fact}[theorem]{Fact}
\newtheorem{lemma}[theorem]{Lemma}

\theoremstyle{definition}
\newtheorem{defn}[theorem]{Definition}
\newtheorem{remark}[theorem]{Remark}

\newcommand{\qftp}{\operatorname{qftp}}

\newcommand{\proofversion}{11}
\numberwithin{equation}{section}
\hypersetup{
  pdftitle={SOP2 = SOP3},
  pdfauthor={Artem Chernikov},
  pdfsubject={Version \proofversion}
}

\title[$\mathrm{SOP}_2=\mathrm{SOP}_3$]{$\mathrm{SOP}_2=\mathrm{SOP}_3$}
\author{Artem Chernikov}

\begin{document}

\begin{abstract}
The classes of $\mathrm{SOP}_2$ and
$\mathrm{SOP}_3$ first-order theories coincide. This answers a question
of D\v{z}amonja and Shelah from 2004 \cite{DzamonjaShelah}.
\end{abstract}
\maketitle

\section{Introduction}
Shelah formulated the strict order property $\mathrm{SOP}$, and later introduced the finite-cycle hierarchy
$\mathrm{SOP}_n$, for $n\geq 3$ \cite{ShelahSOP},
 in the context of model-theoretic classification for ``unordered'' structures beyond stability \cite{Shelah1971}.  He proved $
 \mathrm{SOP}\Longrightarrow\mathrm{SOP}_{n+1}
 \Longrightarrow\mathrm{SOP}_n\Longrightarrow\mathrm{TP}$, 
where a theory is simple precisely when it does not have the tree
property $\mathrm{TP}$ \cite{shelah1980simple}, and gave examples separating these classes.  D\v{z}amonja and Shelah subsequently introduced the tree
configurations $\mathrm{SOP}_2$ and $\mathrm{SOP}_1$ and
proved $
 \mathrm{SOP}_3\Longrightarrow\mathrm{SOP}_2
 \Longrightarrow\mathrm{SOP}_1\Longrightarrow\mathrm{TP}$. 
They asked whether either of the first two implications is reversible 
\cite{DzamonjaShelah} (an example  $T^{\ast}_{\textrm{feq}}$ separating $\mathrm{SOP}_1$ and $\mathrm{TP}$ was given in \cite{ShelahUsvyatsov}; there  is a gap in the proof there  \cite[footnote on p.22]{harrison2013independence},  a different  proof was given in \cite{ChernikovRamsey}).

These properties give strong non-structure results. Shelah proved that every theory with $\mathrm{SOP}_3$ is maximal in Keisler's order \cite{ShelahSOP}, and this was strengthened to the interpretability-order $\triangleleft^*$-maximality by Shelah and Usvyatsov \cite{ShelahSOP, ShelahUsvyatsov}. Conversely, it was shown that, under GCH,  $\triangleleft^*$-maximality implies $\mathrm{SOP}_2$ \cite{DzamonjaShelah, ShelahUsvyatsov}. More recently, 
Malliaris and Shelah relaxed the sufficient hypothesis for $\triangleleft^*$-maximality from $\mathrm{SOP}_3$ to $\mathrm{SOP}_2$, obtaining an equivalence  
\cite{MalliarisShelahCSP, MalliarisShelahApplications}.  On the other hand, theories without these properties admit positive structure theory. For simple  theories \cite{shelah1980simple} it was developed in the 1990s starting with \cite{hrushovski2002pseudo, kim1998forking, kim1997simple} centered around forking-independence (see \cite{wagner2000simple, kim2013simplicity}); and more recently for $\mathrm{NSOP}_1$ theories, centered around Kim-independence \cite{ChernikovRamsey, KaplanRamsey, KaplanRamseyShelah, KaplanRamseyTransitivity, DobrowolskiKimRamsey, ChernikovKimRamsey, KimKimLee}.

The two questions of  D\v{z}amonja and Shelah   were repeatedly highlighted in subsequent work. In a recent breakthrough, Mutchnik proved $\mathrm{SOP}_2 = \mathrm{SOP}_1$
\cite{Mutchnik}.  Partial results around the question $\mathrm{SOP}_2 \stackrel{?}{=} \mathrm{SOP}_3$ were obtained in 
\cite{conant2017axiomatic, KRS, mutchnik2026conant, mutchnik2023properties, MutchnikApplications}. Here we show:

\begin{theorem}\label{thm:main}
 $\mathrm{SOP}_2\Longrightarrow\mathrm{SOP}_3$ (hence $\mathrm{SOP}_1 = \mathrm{SOP}_2=\mathrm{SOP}_3$ by the above).
\end{theorem}

\noindent Combined with the aforementioned results, this establishes a robust dividing line strictly above simplicity and below NSOP. The proof starts with an appropriately indiscernible witness to $\mathrm{SOP}_2$ as in Kaplan, Ramsey and Simon \cite{KRS} (based on the modeling results in \cite{takeuchi2012existence, kim2014tree}), and directly constructs  a $\mathrm{SOP}_3$ formula splitting into two cases via tree indiscernible  manipulation. 

\subsection*{AI disclosure} The proof was found using ChatGPT 5.6 and simplified and
streamlined by the author.

\subsection*{Acknowledgements}
We thank Nick Ramsey for suggesting a correction to the proof of \cite[Lemma~3.8]{KRS}, and Itay Kaplan and Scott Mutchnik for their comments on the preliminary  version of the paper. Chernikov was partially supported by the NSF
Research Grants DMS-2246598, DMS-2554164 and by the Alexander von Humboldt Foundation.

\section{Preliminaries}

\subsection{Trees}
Let $
 \omega^{<\omega}=\bigcup_{n<\omega}\omega^n$, 
 $\omega^{\leq\omega}=\omega^{<\omega}\cup\omega^\omega$. 
For $\eta,\nu\in\omega^{\leq\omega}$, $\eta\unlhd\nu$ means that $\eta$ is a (non-strict) initial segment of
  $\nu$, $\eta\perp\nu$ means that neither sequence is an initial segment
  of the other, $\eta\wedge\nu$ is their longest common initial segment, and $\leq_{\mathrm{lex}}$ is the lexicographic order. We write $\triangleleft$ and $<_{\mathrm{lex}}$ for the strict versions of the corresponding orders. If $\eta\unlhd\nu$, we call $\eta$ an \emph{ancestor} of $\nu$ and
$\nu$ a \emph{descendant} of $\eta$; it is a \emph{proper ancestor}
when $\eta\triangleleft\nu$.
We write $\eta\upharpoonright n$ for the initial segment of length
$n$, $\eta^\frown\nu$ for concatenation, $0^n$ for the sequence of
$n$ zeros, and $0^\omega$ for the constant zero sequence (thus
$0^0$ is the empty sequence).  A \emph{branch} is the
set of finite initial segments of some member of $\omega^\omega$. We regard $\omega^{\leq\omega}$ as an $L_{0,P} = (\unlhd,\wedge,\leq_{\mathrm{lex}},P)$-structure, with
$P(\omega^{\leq\omega})=\omega^\omega$, and call the elements of $P$ \emph{leaves}.   

\subsection{\texorpdfstring{$\mathrm{SOP}_2$ and $\mathrm{SOP}_3$}
                           {SOP2 and SOP3}} Throughout, $T$ is a complete $L$-theory, $\mathbb M \models T$ is a monster model. 
\begin{defn}\label{def: SOP2}
	A formula $\varphi(x;y)$ has $\mathrm{SOP}_2$ if there are tuples
$(b_\eta)_{\eta\in\omega^{<\omega}}$ such that:
\begin{enumerate}
\item for every $\sigma\in\omega^\omega$, the set $
    \{\varphi(x;b_{\sigma\upharpoonright n}):n<\omega\}$ 
  is consistent;
\item whenever $\eta,\nu\in\omega^{<\omega}$ satisfy
  $\eta\perp\nu$, the pair
  $\{\varphi(x;b_\eta),\varphi(x;b_\nu)\}$ is inconsistent.
\end{enumerate}
The theory $T$ has $\mathrm{SOP}_2$ if one of its formulas does. (The original definition \cite{DzamonjaShelah} used the tree $2^{<\omega}$ instead of $\omega^{<\omega}$, but this is easily equivalent, see e.g.~\cite{kim2011notions}.)
\end{defn}

\begin{defn}\label{def: SOP3 rel}\cite{ShelahSOP} 
	A formula $Q(z;z')$, with $z$ and $z'$ of the same 
sort, has $\mathrm{SOP}_3$ if
there are tuples $(d_i)_{i<\omega}$ such that $\models Q(d_i;d_j)$ for every $i<j$,
but  $
 Q(z_0;z_1)\wedge Q(z_1;z_2)\wedge Q(z_2;z_0)$ 
is inconsistent.   The theory $T$ has $\mathrm{SOP}_3$ if
such a formula exists.
\end{defn}

\noindent The following is a variant of \cite[Fact 6.3]{Mutchnik} (which in turn corrects \cite[Proposition 7.2]{conant2017axiomatic}).
\begin{lemma}\label{lem:triangle-criterion}
Let $\alpha(v;p), \beta(v;p)$ be formulas and 
$(v_i,p_i)_{i<\omega}$ tuples so that for all $i<j$,
\begin{enumerate}
\item $\models\alpha(v_i;p_j)\wedge\beta(v_j;p_i)$;
\item $\{\alpha(v;p_i),\beta(v;p_j)\}$ is inconsistent.
\end{enumerate}
Then $T$ has $\mathrm{SOP}_3$.
\end{lemma}

\begin{proof}
Define $
 R((v,p);(v',p')):=
 \neg\exists w\bigl(\alpha(w;p)\wedge\beta(w;p')\bigr)
 \wedge\alpha(v;p')\wedge\beta(v';p)$. 
The two hypotheses give
$\models R((v_i,p_i);(v_j,p_j))$ for all $i<j$.  If three tuples
$z_i=(v_i,p_i)$, for $i<3$, satisfied $ R(z_0;z_1)\wedge R(z_1;z_2)\wedge R(z_2;z_0)$, they would give, respectively, $
 \alpha(v_0;p_1)$, $
 \neg\exists w\bigl(\alpha(w;p_1)\wedge\beta(w;p_2)\bigr)$, $\beta(v_0;p_2)$ --- 
a contradiction.  Thus $R$ witnesses $\mathrm{SOP}_3$.
\end{proof}

 \subsection{Treetop indiscernibles}
\begin{defn}
	An array $(a_\eta)_{\eta\in\omega^{\leq\omega}}$ (where $a_\eta$ are tuples in $\mathbb{M}$, possibly of different sorts)  is a \emph{treetop indiscernible} if for any finite tuples $\bar\eta, \bar\nu$ from $\omega^{\leq\omega}$, $
 \qftp_{L_{0,P}}(\bar\eta)=\qftp_{L_{0,P}}(\bar\nu) \ \Rightarrow \ 
 a_{\bar\eta}\equiv a_{\bar\nu}$. 
Here $a_{\bar\eta} = (a_{\eta_i} : i < n)$ for $\bar \eta = (\eta_i : i < n)$, and
$a_{\bar\eta}\equiv a_{\bar\nu}$ means that the two tuples have the same type in $\mathbb{M}$.
\end{defn}

\begin{defn}
	 We say that $(a_\eta)_{\eta\in\omega^{\leq\omega}}$ is
\emph{locally based on} $(e_\eta)_{\eta\in\omega^{\leq\omega}}$ if, for every index tuple
$\bar\eta$ and  formula $\theta$, there is an index tuple
$\bar\nu$ such that $
 \qftp_{L_{0,P}}(\bar\eta)=
 \qftp_{L_{0,P}}(\bar\nu)$ 
and $a_{\bar\eta}$ and $e_{\bar\nu}$ agree on $\theta$.  (In particular, if $\theta$ is true of every tuple
$e_{\bar\nu}$ having some fixed quantifier-free index type, then  $\theta$ is true of every corresponding
tuple $a_{\bar\eta}$.)
\end{defn}

\begin{fact}\cite[Lemma~3.8]{KRS}
\label{fact:modeling}
Given any array
$(e_\eta)_{\eta\in\omega^{\leq\omega}}$ of tuples, there is a treetop indiscernible
$(a_\eta)_{\eta\in\omega^{\leq\omega}}$ locally based on
$(e_\eta)$.
\end{fact}
\begin{remark}
\label{rem:modeling-gap}
The proof  of
\cite[Lemma~3.8]{KRS} has a gap: when the coordinates in $P$ are removed, Ramsey's
theorem is applied to the remaining nonleaf skeleton, and an arbitrary
copy of that skeleton is then assumed to admit the required leaves.
This need not hold in $\omega^{\leq\omega}$ (the meet-closed configuration $\emptyset, \langle0\rangle, 
 \langle1,0,0,\ldots\rangle\in P, \langle2\rangle$ has a nonleaf skeleton whose copy
$(\emptyset,\langle0\rangle,\langle1\rangle)$ has no completing leaf). 
To fix it, for the finite structure $C$  in the
compactness argument, choose $r>|P(C)|$ and let $
 H:=h(\omega^{<\omega})\subseteq\omega^{<\omega}$, $h(\eta)(i):=r(\eta(i)+1)$, then  $H\cong\omega^{<\omega}$ as $L_0$-structures. Every $\bar\mu\in q_{-,i}(H)$ admits an extension
$\bar\zeta\in q_i(\omega^{\leq\omega})$ with
$\bar\zeta_-=\bar\mu$: at every node, two consecutive successor
directions belonging to $H$ have 
$\geq r-1$  directions in $\omega^{<\omega}$ between them, and there are $\geq r$ unused directions below the first one and above the last one. Thus the  required leaves can be
inserted in their prescribed $\leq_{\mathrm{lex}}$-cuts; this preserves all meets and $\leq_{\mathrm{lex}}$. Hence the restricted colorings $c_{-,i}$ are defined on $q_{-,i}(H)$ as in
\cite[Lemma~3.8]{KRS} and are well-defined by its Lemma~3.7 and the
$s$-indiscernibility of the source tree. As $H\cong\omega^{<\omega}$, the same Ramsey theorem applied inside
$H$ yields a copy $C'_-\cong C_-$ on which all the restricted colorings
are constant, and we can extend $C'_-$ to
an $L_{0,P}$-copy $C'\cong C$ by the above. 
   \end{remark}

\begin{fact} \cite[Lemma~7.9]{KRS}\label{lem:treetop-witness}
	If $\varphi(x;y)$ has $\mathrm{SOP}_2$, then there is 
a treetop indiscernible array 
$(a_\eta)_{\eta\in\omega^{\leq\omega}}$ satisfying:
\begin{enumerate}
\item if $\eta,\nu\in\omega^{<\omega}$ are incomparable, then
  $\{\varphi(x;a_\eta),\varphi(x;a_\nu)\}$ is inconsistent;
\item if $\sigma\in\omega^\omega$ and $\eta\triangleleft\sigma$, then
  $\mathbb M\models\varphi(a_\sigma;a_\eta)$.
\end{enumerate}

\end{fact}


\section{\texorpdfstring{The implication
$\mathrm{SOP}_2\Rightarrow\mathrm{SOP}_3$}
{The implication SOP2 implies SOP3}}

\begin{lemma}\label{lem:transport}
Fix $m<\omega$ and $1\leq k<\omega$, and let
\[
 \lambda_t := 0^t{}^\frown1^\frown0^\omega\quad(t\leq m),
 \qquad \rho_l := 0^{m+1+l}\quad(l<k).
\]
There are leaves $\lambda_t^a$ and nonleaves $\mu^a$ and 
$\rho_l^a$,  for all $a<\omega$, $t\leq m$ and $l<k$, such that, for every $a<b < \omega$:
\begin{enumerate}
\item $\mu^a\triangleleft\lambda_t^b$ for all $t\leq m$;
\item $\rho_l^a\perp\mu^b$ for all $l<k$;
\item $
 \qftp_{L_{0,P}}\bigl((\lambda_t^a)_{t\leq m},
                         (\rho_l^b)_{l<k}\bigr)
 =
 \qftp_{L_{0,P}}\bigl((\lambda_t)_{t\leq m},
                         (\rho_l)_{l<k}\bigr)
$. \label{eq:qftp-transport}
\end{enumerate}
\end{lemma}

\begin{proof}
For $a<\omega$, define
\begin{gather*}
	 \mu^a :=0^{(m+1)a}, \qquad  \lambda_t^a:=0^{(m+1)a+t}{}^\frown1^\frown0^\omega, \qquad \rho_l^a :=0^{(m+1)a+m}{}^\frown1^\frown0^l.
\end{gather*}
If $a<b$, then $(m+1)b\geq(m+1)a+m+1$, so $\mu^a$ is a strict
initial segment of every $\lambda_t^b$.  Also, $\rho_l^a$ has a $1$
in coordinate $(m+1)a+m$, whereas $\mu^b$ is defined there and has
value $0$.  Hence $\rho_l^a\perp\mu^b$, proving (1) and (2).

For (3), fix $a<b$, put $h_t=0^t$ and
$h_t'=0^{(m+1)a+t}$, and consider the map
\[
 \lambda_t\longmapsto\lambda_t^a,\qquad
 \rho_l\longmapsto\rho_l^b,\qquad
 h_t\longmapsto h_t'.
 \tag{3.2}\label{eq:closure-map}
\]
The meet-closures of the two tuples in (3) consist exactly of the
elements displayed in \eqref{eq:closure-map}.  Prefixing by
$0^{(m+1)a}$ preserves all meets among the $h_t$ and $\lambda_t$;
the $\rho_l$ and $\rho_l^b$ each form a strictly increasing chain under
$\unlhd$; and the mixed meets are
\[
 h_t\wedge\rho_l=h_t,\qquad
 \lambda_t\wedge\rho_l=h_t,
 \qquad
 h_t'\wedge\rho_l^b=h_t',\qquad
 \lambda_t^a\wedge\rho_l^b=h_t'.
\]
Thus \eqref{eq:closure-map} preserves $\wedge$.  It also preserves
$P$, which holds exactly on the $\lambda$-terms.  Finally, the full
lexicographic order on the first closure is
\[
 h_0<_{\rm lex}\cdots<_{\rm lex}h_m
 <_{\rm lex}\rho_0<_{\rm lex}\cdots<_{\rm lex}\rho_{k-1}
 <_{\rm lex}\lambda_m<_{\rm lex}\cdots<_{\rm lex}\lambda_0,
\]
and its image has the same order: because
$(m+1)b\geq(m+1)a+m+1$, every $\rho_l^b$ is still $0$ at the branching coordinate of
$\lambda_m^a$.  Since
$x\unlhd y$ is equivalent to $x\wedge y=x$, the map is an
$L_{0,P}$-isomorphism of the generated substructures.  This proves
\eqref{eq:qftp-transport}.
\end{proof}

\begin{theorem}\label{thm:hard}
If $T$ has $\mathrm{SOP}_2$, then $T$ has $\mathrm{SOP}_3$.
\end{theorem}

\begin{proof}
Apply Fact~\ref{lem:treetop-witness} to a formula $\varphi(x,y)$ witnessing
$\mathrm{SOP}_2$, and let  
$(a_\eta)_{\eta\in\omega^{\leq\omega}}$ be the resulting 
treetop indiscernible.
Define $
 I(y,y'):=\neg\exists x
 \bigl(\varphi(x,y)\wedge\varphi(x,y')\bigr)$. 
For $j,t<\omega$, let 
\[
 \lambda_t := 0^t{}^\frown1^\frown0^\omega,
 \qquad s_j := a_{0^j},
 \qquad c_t := a_{\lambda_t}.
\]
Since $0^j\triangleleft\lambda_t$ whenever $j\leq t$,
Fact~\ref{lem:treetop-witness}(2) gives
\[
 \mathbb M\models\varphi(c_t,s_j)
 \quad\text{for all }j\leq t.
 \tag{3.4}\label{eq:ancestor-instance}
\]

For $i<\omega$, consider the partial type
\[
 \Gamma_i(y)=
 \{\varphi(c_t,y):t\leq i\}
 \cup\{I(s_j,y):j>i\}.
\]
The proof splits into two exhaustive cases.

\medskip
\noindent\emph{Case 1: $\Gamma_i$ is consistent for all $i < \omega$.}
Choose $v_i\models\Gamma_i$, put $p_i=(s_i,c_i)$, and define
\[
 \alpha(v;y_0,y_1):=I(y_0,v),
 \qquad \beta(v;y_0,y_1):=\varphi(y_1,v).
\]
If $i<j$, the definitions of $\Gamma_i$ and $\Gamma_j$ give
$\alpha(v_i;p_j)$ and $\beta(v_j;p_i)$.  And  $
 \{\alpha(v;p_i),\beta(v;p_j)\}
 =\{I(s_i,v),\varphi(c_j,v)\}$ 
is inconsistent: by \eqref{eq:ancestor-instance},
$ \models \varphi(c_j,s_i)$, so any realization $v$ would make $c_j$ witness
the negation of $I(s_i,v)$.  Lemma~\ref{lem:triangle-criterion}
therefore gives $\mathrm{SOP}_3$.

\medskip
\noindent\emph{Case 2: $\Gamma_m$ is inconsistent for some $m$.}
Its positive part $\{\varphi(c_t,y):t\leq m\}$ is realized by $s_0$
by \eqref{eq:ancestor-instance}.  Compactness therefore gives a
finite inconsistent subfamily of $\Gamma_m$ containing at least one
formula of the form $I(s_j,y)$.  Add all omitted positive formulas,
and let $N>m$ be the largest index occurring among its formulas of
the form $I(s_j,y)$.  Adding the finitely many formulas $I(s_j,y)$
with $m<j\leq N$ preserves inconsistency.  Setting $k=N-m$, we obtain
\begin{equation}
 \models\neg\exists y\left(
   \bigwedge_{t\leq m}\varphi(c_t,y)
   \wedge\bigwedge_{l<k}I(s_{m+1+l},y)
 \right).
 \tag{3.6}\label{eq:canonical-obstruction}
\end{equation}

Apply Lemma~\ref{lem:transport} to $m,k$.  Let 
$p=(\bar x,\bar u)$, where
$\bar x=(x_t)_{t\leq m}$ and $\bar u=(u_l)_{l<k}$, and 
\begin{gather*}
	\alpha(v;p):=\bigwedge_{t\leq m}\varphi(x_t,v),
 \qquad
 \beta(v;p):=\bigwedge_{l<k}I(u_l,v).
\end{gather*}
For $i<\omega$, let 
\[
 v_i:=a_{\mu^i},\qquad
 p_i:=\left((a_{\lambda_t^i})_{t\leq m},
            (a_{\rho_l^i})_{l<k}\right).
\]
Fix $i<j$.  By Lemma~\ref{lem:transport}(1), $\mu^i$ is a proper
ancestor of every $\lambda_t^j$. Fact~\ref{lem:treetop-witness}(2)
therefore gives 
$\models \varphi(a_{\lambda_t^j}, a_{\mu^i})$ for every $t\leq m$, so $\models \alpha(v_i;p_j)$. By Lemma~\ref{lem:transport}(2), every $\rho_l^i, l < k$
is incomparable with $\mu^j$, so Fact~\ref{lem:treetop-witness}(1)
therefore gives $\models I(a_{\rho_l^i},a_{\mu^j})$ for every $l<k$, so  $
 \models\beta(v_j;p_i)$.  Combining, $
 \models \alpha(v_i;p_j) \land \beta(v_j;p_i)$. 

On the other hand, Lemma \ref{lem:transport}(3) and treetop indiscernibility give
\[
 \left((a_{\lambda_t^i})_{t\leq m},
       (a_{\rho_l^j})_{l<k}\right)
 \equiv
 \left((c_t)_{t\leq m},(s_{m+1+l})_{l<k}\right).
\]
Combined with \eqref{eq:canonical-obstruction}, this says exactly that
$\{\alpha(v;p_i),\beta(v;p_j)\}$ is inconsistent.  All hypotheses of
Lemma~\ref{lem:triangle-criterion} hold, so $T$ has
$\mathrm{SOP}_3$.
\end{proof}

%
%

\bibliographystyle{plain}
\bibliography{SOP3}

\end{document}